\documentclass[12pt]{article}
\RequirePackage[colorlinks,citecolor=blue,urlcolor=blue,linkcolor=blue]{hyperref}
\hypersetup{
colorlinks = true,
citecolor=blue,
urlcolor=blue,
linkcolor=blue,
pdfauthor = {Alexey Kuznetsov},
pdfkeywords = {Dirichlet L-function, Siegel function, Gaussian quadrature, high-precision algorithm},
pdftitle = {Simple and accurate approximations to Dirichlet L-functions},
pdfpagemode = UseNone
}
\usepackage[titletoc,title]{appendix}
\usepackage{graphicx,xspace,colortbl}
\usepackage{epstopdf}
\usepackage{amsmath,amsthm,amsfonts,amssymb}
\usepackage{color}
\usepackage{enumerate}
\usepackage{fancybox}
\usepackage{epsfig}
\usepackage{subfig}
\usepackage{pdfsync}
\usepackage{comment}
    \def\qed{\hfill$\sqcap\kern-8.0pt\hbox{$\sqcup$}$\\}
    \def\re{\textnormal {Re}}

    \def\r{{\mathbb R}}
    \def\c{{\mathbb C}}
    \def\d{{\textnormal d}}
    \def\i{{\textnormal i}}

	\newtheorem{proposition}{Proposition}
	
	\theoremstyle{definition}

\title{Simple and accurate approximations to Dirichlet $L$-functions}
\author{ 
{Alexey Kuznetsov\footnote{Dept. of Mathematics and Statistics,  York University,
4700 Keele Street, Toronto, ON, M3J 1P3, Canada.  
  Email: akuznets@yorku.ca}} }

\date{\today}

\begin{document}
\maketitle

\begin{abstract}
We extend the quadrature method developed in \cite{Kuznetsov_2026} to Dirichlet $L$-functions associated with
primitive characters. The approximations combine finite Dirichlet sums with correction terms 
whose quadrature coefficients can be precomputed for each character. 
Numerical experiments demonstrate high accuracy in the critical strip and in other vertical strips.
We also describe how to combine these  approximations with direct summation and Euler–Maclaurin
summation to evaluate $L(s,\chi)$ throughout the complex plane.
\end{abstract}
{\vskip 0.15cm}
{\vskip 0.25cm}

\section{Introduction and main results}\label{section:Intro}

Throughout the paper, $q \ge 3$ is an integer, $\chi$ is a  primitive Dirichlet character modulo $q$, and $s=\sigma+\i t$. The Dirichlet $L$-function $L(s,\chi)$ is an entire function, defined via 
\[
L(s,\chi):=\sum\limits_{n\ge 1} \frac{\chi(n)}{n^s}
\]
in the half-plane $\sigma>0$ and via a functional equation
 \cite[Eq. (25.15.5)]{NIST}
\begin{equation}\label{reflection_formula}
L(s,\chi)=X_{\chi}(s) L(1-s,\bar \chi)
\end{equation}
in $\sigma<1$. Here
\begin{equation}\label{def_X_chi}
X_{\chi}(s):=\frac{\tau(\chi)}{i^{\kappa} \pi} 
\Big( \frac{2\pi}{q} \Big)^s \Gamma(1-s) \sin\Big(\frac{\pi}{2} (s+\kappa)\Big),
\end{equation}
where $\kappa:=(1-\chi(-1))/2 \in \{0,1\}$ and  
\[
\tau(\chi):=\sum\limits_{n=1}^q \chi(n) e^{2\pi \i n/q}.
\]

Our goal is to construct approximations to $L(s,\chi)$ in fixed vertical strips $a\le\sigma\le b$ for large positive $t$. For fixed $q$ and a fixed number of quadrature nodes, their evaluation requires $O(t^{1/2})$ numerical operations, excluding precomputation. Their accuracy is assessed numerically; we do not establish rigorous bounds for the quadrature error in this paper.

Methods for computing Dirichlet $L$-functions are closely related to those for the Riemann zeta function \cite{Rubinstein_2005}. For fixed modulus and working precision, standard Euler--Maclaurin methods \cite{BERNDT1975,Johansson_2015} use $O(|s|)$ terms, whereas Riemann--Siegel methods \cite{Davies_1965,Deuring_1967} use $O(|t|^{1/2})$ terms. For fixed $q$, Hiary's algorithm \cite[Theorem 1.1]{Hiary_2014} gives a complexity of $O_q((1+|t|)^{1/3+o(1)})$ in the critical strip.

Our construction extends the quadrature algorithm for the Riemann zeta function in \cite{Kuznetsov_2026}, which was motivated by Galway's method \cite{Galway_2001}. Tyagi \cite{Tyagi_2022} developed a related approach to high-precision evaluation of Dirichlet $L$-functions using double-exponential quadrature.

Let $p\ge 2$ be an integer. Our approximations require $2pq$ precomputed complex coefficients
\[
\big \{\omega^{(i)}_{\chi,p,j}, \; \lambda^{(i)}_{\chi,p,j} \; : \; 0\le i < q, \; 1\le j \le p \big \}.
\]
Section \ref{section2} describes how these coefficients are computed. For $t>0$, define
\begin{equation}\label{def_N_r_k_M}
N:=\Big\lfloor  \sqrt{\frac{qt}{2\pi}} \Big\rfloor, \;\;\; M:=N+1/2,\;\;\;
r:=N \; {\textnormal{mod}} \; q, \;\;\; k:=(N-r)/q.
\end{equation}
For a complex-valued function $f$, write $\bar f(s):=\overline{f(\bar s)}$. We use principal branches of logarithms and powers throughout. Our approximation to $L(s,\chi)$ is
 \begin{equation}\label{L(s,chi)_approximation}
 L_p(s,\chi):=f_{\chi,p}(s) + X_{\chi}(s) \bar f_{\chi,p}(1- s), 
 \end{equation}
 where 
  \begin{equation}\label{f_chi_approximation}
 f_{\chi,p}(s):=\sum\limits_{n=1}^N \frac{\chi(n)}{n^s} -(-1)^k 
M^{-s} \sum\limits_{j=1}^p \omega^{(r)}_{\chi,p,j} 
e^{-\frac{2 \pi}{q} M \lambda^{(r)}_{\chi,p,j}}
\bigg(1+\i \frac{\lambda^{(r)}_{\chi,p,j}}{M}  \bigg)^{-s}.
 \end{equation} 
In \eqref{L(s,chi)_approximation}, $\bar f_{\chi,p}(1-s)=\overline{f_{\chi,p}(1-\bar s)}$. Since $s$ and $1-\bar s$ have the same imaginary part, both terms in \eqref{L(s,chi)_approximation} use the same $N$, $M$, $r$, and $k$.

Following \cite{Kuznetsov_2026}, we measure the error on a finite horizontal grid by computing
\[
\Delta_{\chi,p}(t;a,b):=\max\limits_{0\le j \le 100}
\big | L_p(\sigma_j+\i t,\chi)-L(\sigma_j+\i t,\chi) \big|, \;\;\;
{\textnormal{where}} \; \;\; \sigma_j=a+(b-a)j/100. 
\]
Figure \ref{fig1} plots $\Delta_{\chi,p}(t;1/2,5)$ on a logarithmic scale for the character $\chi_{5,2}$ in the Conrey labeling used by the LMFDB \cite{lmfdb}.   In the experiments shown, the error decreases rapidly as $p$ increases. Experiments with other primitive characters exhibited similar behavior.

\begin{figure}[t!]
\centering
\includegraphics[width =16.5cm]{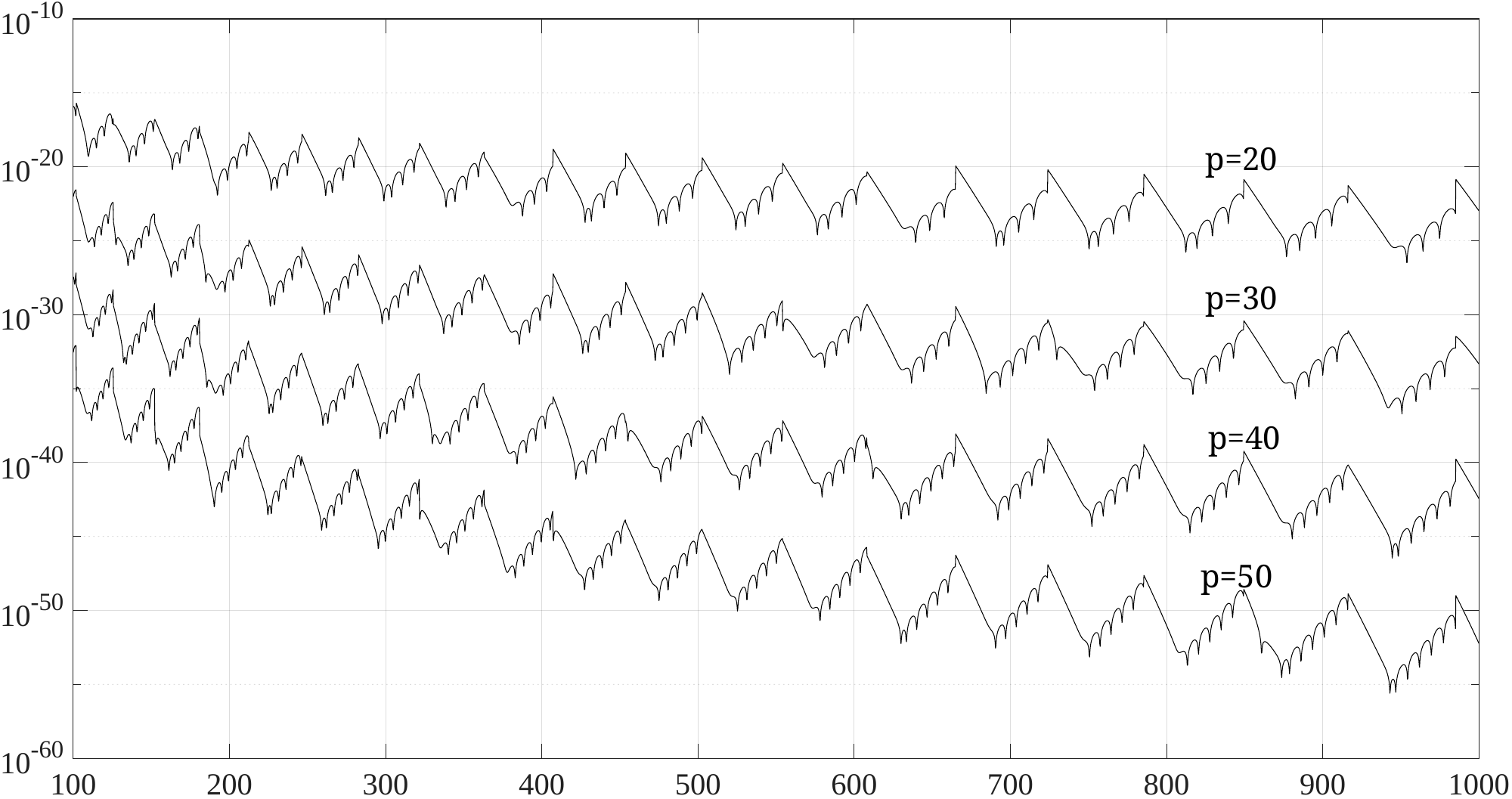} 
\caption{The errors $\Delta_{\chi,p}(t;1/2,5)$ for $\chi=\chi_{5,2}$ and $p=20,30,40,50$. } 
\label{fig1}
\end{figure}

\vspace{0.25cm}
\noindent
{\bf Remark 1:}  The parameter $p$ counts all quadrature nodes in the present paper. In \cite{Kuznetsov_2026}, it counts only the nodes in the fourth quadrant: the complete quadrature rule has $2p+1$ nodes, consisting of those nodes, their negatives, and the origin. If the construction is formally extended to the principal character $\chi_0$ modulo $1$, then the corresponding weights here are one half of those in \cite{Kuznetsov_2026}, while the nodes and interpolation points agree. With this normalization, $\zeta_p(s)=L_{2p+1}(s,\chi_0)$ (though this formal comparison lies outside the standing assumption $q\ge3$).

\vspace{0.25cm}
 We precomputed coefficients for all primitive Dirichlet characters with $3\le q\le10$ and for $p\in\{20,30,40,50,70,100\}$, using Fortran 90 and David Bailey's MPFUN2020 arbitrary-precision package \cite{Bailey_2020}: these coefficients can be downloaded from \href{https://github.com/Alexey-Kuznetsov-math/Dirichlet-L-functions}{GitHub} or from the author's \href{https://kuznetsovmath.ca/}{webpage}. Implementation details and numerical validation are discussed in Section \ref{section2}. For these computed coefficient sets, we observed the following symmetries (the statements below are observations for those sets, rather than unconditional assertions for every $p$ and $\chi$):
\begin{itemize}
\item[(i)] For $0\le r<q$ and $1\le j \le p$ we have 
\begin{equation}\label{lambda_omega_symmetry1}
\lambda^{(q-1-r)}_{\chi,p,j}=-\lambda^{(r)}_{\chi,p,p+1-j},
\;\;\; 
\omega^{(q-1-r)}_{\chi,p,j}=\chi(-1)\omega^{(r)}_{\chi,p,p+1-j}.
\end{equation}
\item[(ii)] When $4|q$ we have for all $0\le r <q/2$ and $1\le j\le p$
\begin{equation}\label{lambda_omega_symmetry2}
\lambda^{(r+q/2)}_{\chi,p,j}=\lambda^{(r)}_{\chi,p,j}, \;\;\;
\omega^{(r+q/2)}_{\chi,p,j}=\i (-1)^r \omega^{(r)}_{\chi,p,j}. 
\end{equation}
\end{itemize}

 Proposition \ref{prop_symmetry} in Section \ref{section2} proves the corresponding kernel symmetries. For the even values of $p$ considered here, coefficient sets satisfying \eqref{lambda_omega_symmetry1} can be stored using $pq$ complex numbers; when $4\mid q$, using \eqref{lambda_omega_symmetry2} as well reduces this count to $pq/2$. These counts include the symmetry within the residue class $r=(q-1)/2$ when $q$ is odd. Since storage still grows with $q$, our method is particularly well suited to repeated evaluations for a fixed character of moderate modulus.

This paper is organized as follows. Section \ref{section2} derives the approximations, evaluates the required moments, describes the construction of the coefficients, and proves the kernel symmetries. Section \ref{Section_computing_L} combines the approximations with direct  summation and Euler--Maclaurin summation to construct and algorithm for computing $L(s,\chi)$ in the entire complex plane and discusses implementation and validation.

\section{Deriving the approximations and computing 
$\omega^{(r)}_{\chi,p,j}$ and $\lambda^{(r)}_{\chi,p,j}$ }\label{section2}

The construction follows the approach in \cite[Section 2]{Kuznetsov_2026}. Set $\theta:=e^{-\pi\i/4}$ and assume $t>0$. We begin with Siegel's representation \cite{Siegel_1943} (see also \cite[Section 4]{Tyagi_2022}):
\begin{equation*}
L(s,\chi)=f_{\chi}(s) + X_{\chi}(s) \bar f_{\chi}(1- s), \;\;\; s\in \c,
\end{equation*}
where $X_\chi$ is defined in \eqref{def_X_chi} and the entire function $f_\chi$ is given by
\begin{equation}\label{eqn1}
f_{\chi}(s):=\frac{1}{2\pi \i} \int_{0 \swarrow 1} 
e^{\frac{\pi \i}{q} u^2} W_{\chi}(u) u^{-s} \d u, 
\end{equation}
and 
\begin{equation}\label{def_W_chi}
W_{\chi}(u):=\frac{\pi}{2q} \sum\limits_{n=1}^{2q} \chi(n) e^{-\frac{\pi \i n^2}{q}} \cot\big(\tfrac{\pi}{2q} (u-n)\big). 
\end{equation}
The contour of integration in \eqref{eqn1} is a straight line oriented in the southwest direction  and crossing the real axis between $0$ and $1$. We shift the contour parallel to itself so that it crosses the real axis at $M$, as defined in \eqref{def_N_r_k_M}. The integer points crossed are $u=1,\ldots,N$, with residues
\[
{\textnormal{Res}}_{\,u=n\, }\Big( e^{\frac{\pi \i}{q} u^2} W_{\chi}(u) u^{-s} \Big)=\frac{\chi(n)}{n^s}. 
\]
The residue is zero when $\chi(n)=0$, in which case there is no pole. Using $W_\chi(u+q)=(-1)^qW_\chi(u)$ and substituting $u=M+\sqrt q\,x/\theta$ after the contour shift gives 
\begin{equation}\label{eqn_f_chi_main}
f_{\chi}(s)=\sum\limits_{n=1}^N \frac{\chi(n)}{n^s} -
 (-1)^k M^{-s}  \int_{\r} 
e^{-\pi x^2 -2 \pi M \theta x/\sqrt{q}}
\Big(1+\frac{\sqrt{q}x}{\theta M}  \Big)^{-s} K_{\chi,r} (x) \d x,
\end{equation}
where we denoted
\begin{equation}\label{def_K_chi_r}
K_{\chi,r}(x):=\frac{\theta\sqrt{q}}{2\pi} 
e^{\frac{\pi \i}{q}(r+1/2)^2} 
W_{\chi}\Big(r+\frac{1}{2} + \frac{\sqrt{q}x}{\theta} \Big).
\end{equation}

Define 
\begin{equation}\label{def_I_chi_r}
{\mathcal I}_{\chi,r}(s,M):=\int_{\r} 
e^{-\pi x^2 -2 \pi M \theta x/\sqrt{q}}
\Big(1+\frac{\sqrt{q}x}{\theta M}  \Big)^{-s} K_{\chi,r} (x) \d x.
\end{equation}
The above formulas reduce evaluation of $L(s,\chi)$ to evaluation of $\mathcal I_{\chi,r}(s,M)$. Now we fix $\chi$, $p$, and a residue class $r\in\{0,\ldots,q-1\}$; $M$ and $N$ continue to be chosen from $s$ by \eqref{def_N_r_k_M}. Following \cite[Eq.\ (10)]{Kuznetsov_2026}, write
\begin{equation*}
{\mathcal I}_{\chi,r}(s,M)=\int_{\r} 
e^{g(x;s)} \eta(\d x),
\end{equation*}
where
\[
g(x)=g(x;s):=-2 \pi M \theta x/\sqrt{q}- s 
\ln\Big(1+\frac{\sqrt{q}x}{M \theta }\Big) \;\;\;
\textnormal{and} \;\;\; \eta(\d x):=e^{-\pi x^2} K_{\chi,r} (x) \d x.
\]

We seek weights $\omega_j=\omega^{(r)}_{\chi,p,j}$ and nodes $x_j=\lambda^{(r)}_{\chi,p,j}/(\theta\sqrt q)$ for the quadrature rule
\begin{equation}\label{eqn_quadrature_integral}
\int_{\r} f(x) \eta(\d x) \approx \sum\limits_{j=1}^p \omega_j f(x_j). 
\end{equation}
The above equation should be compared with \cite[Eq. (11)]{Kuznetsov_2026}. The coefficients are determined by requiring this rule to be exact for the following $2p$ test functions:
\[
f_k(x)=e^{-\pi x^2 - 2\pi \theta x y_{p,k}/\sqrt{q}},
\]
where the interpolation points are equally spaced in $[-1,1]$:
\[
y_{p,k}:=-1+\frac{2k}{2p-1}, \;\;\;  k=0,\ldots,2p-1.
\]
To motivate this choice, rewrite $g$ exactly as
\begin{equation}\label{def_I_chi_r2}
g(x;s)=-\pi x^2+ 
 A \Big(\frac{x^2}{2M^2}-
\frac{\theta x}{M \sqrt{q}} \Big) - s {\mathcal L}_2
\Big(\frac{\sqrt{q} x}{\theta M}\Big),
\end{equation}
where  $A=A(s):=2\pi M^2 + \i q s$ and
\[
{\mathcal L}_2(z):=\ln(1+z)-z+\frac{z^2}{2}=\sum\limits_{n\ge 3} (-1)^{n+1} \frac{ z^n}{n},\;\;\; |z|<1. 
\]
Set
\[
B=B(s):=2 \Big( M - \sqrt{\frac{qt}{2\pi}} \Big). 
\]
Equation \eqref{def_N_r_k_M} gives $-1<B\le1$. Since $qt=2\pi(M-B/2)^2$, we have, for fixed $q$ and bounded $\sigma$,
\begin{equation}\label{A_asymptotics}
A=2 \pi M B + O(1). 
\end{equation}
Because $\mathcal L_2(z)=O(z^3)$ as $z\to0$, the last term in \eqref{def_I_chi_r2} is $O(t^{-1/2})$ for each fixed $x$. Also, \eqref{A_asymptotics} gives
\[
A\Big(\frac{x^2}{2M^2}-
\frac{\theta x}{M \sqrt{q}} \Big) = - \frac{2 \pi \theta x B}{\sqrt{q}}+O(t^{-1/2}),  
\]
as $t\to +\infty$. Consequently, uniformly for $x$ on compact subsets of $\c$,
\[
g(x;s)= -\pi x^2 - \frac{2\pi \theta x B}{\sqrt{q}} + O(t^{-1/2}). 
\]
This local asymptotic calculation suggests approximating $\exp(g(x;s))$ by the family $\exp(-\pi x^2-2\pi\theta xy/\sqrt q)$, $-1\le y\le1$, and motivates the chosen test functions. 

Thus, we arrive at the following problem: we seek $2p$ complex numbers $\{\omega_j\}_{1\le j \le p}$ and $\{x_j\}_{1\le j \le p}$ satisfying
\begin{equation}\label{quadrature_condition0}
\int_{\r} 
e^{-2\pi x^2 -2 \pi y_{p,k} \theta x/\sqrt{q}}
 K_{\chi,r} (x) \d x= \sum\limits_{j=1}^p
 \omega_j e^{-\pi x_j^2-2 \pi y_{p,k} \theta x_j/\sqrt{q}}
\end{equation}
for $k=0,1,\dots,2p-1$. 
The above quadrature condition should be compared with  \cite[Eq. (16)]{Kuznetsov_2026}. We note that $\omega_j$ and $x_j$ also depend on $\chi$, $p$ and $r$, but we suppress this dependence to simplify the notation. 

The integral on the left-hand side of \eqref{quadrature_condition0} can be evaluated explicitly: it can be reduced to a finite linear combination of Mordell integrals \cite{Mordell}. It can also be evaluated directly in the same way as the integrals in \cite[Eq. (5)]{Deuring_1967} or \cite[Section 3]{Siegel_1943}. Define
\begin{equation}\label{def_H_chi_r}
H_{\chi,r}(y):=
\int_{\r} 
e^{-2\pi x^2 -2 \pi y \theta x/\sqrt{q}}
 K_{\chi,r} (x) \d x.
\end{equation} Set $b:=r+1/2$, and let $C_q(u):=\cos(\pi u)$ for odd $q$ and $C_q(u):=\i\sin(\pi u)$ for even $q$. For integers $j$ with $|j|\le q/2$, define
\[
c_j:=\frac{\sqrt{2/q}}{\theta} e^{\frac{\pi \i}{q} (b^2+(q+2j)b)} \sum\limits_{n=1}^q 
\chi(n) e^{-\frac{\pi \i}{q} (n^2+(q+2j)n)}.
\]
Set
\begin{align}
\label{def_P}
P(x)&:=e^{-\frac{\pi \i q}{8}} \sideset{}{'}\sum\limits_{|j|\le q/2} c_j e^{-\frac{\pi \i}{2q}(x+j)^2 }
C_q((x-j)/2), \\
\nonumber
Q(x)&:=\sum\limits_{n=1}^q \chi(r+n) \exp
\Big[ \frac{\pi \i}{q} \Big( (n-1/2)^2-2b(n-1/2) +2x(n-(q+1)/2) \Big) \Big].
\end{align}
The prime in the sum notation in \eqref{def_P} indicates that the two endpoint terms, $j=\pm q/2$, have weight $1/2$ when $q$ is even. Then
\begin{equation}\label{eqn_Mordell}
H_{\chi,r}(y)=\frac{(-1)^q Q(y)-P(y)}{2C_q(y)}.
\end{equation}
Equation \eqref{eqn_Mordell} holds wherever $C_q(y)\ne0$ and extends to all $y\in\c$ by continuity. The integral \eqref{def_H_chi_r} shows that $H_{\chi,r}$ is entire. Since the zeros of $C_q$ are simple, the value of $H_{\chi,r}$ at such a zero $y_0$ is
\[
H_{\chi,r}(y_0)=\frac{(-1)^qQ'(y_0)-P'(y_0)}{2C_q'(y_0)}.
\]

In order to solve the quadrature problem \eqref{quadrature_condition0}, we follow the method in \cite[Section 2]{Kuznetsov_2026}. 
We denote $\mu_{p,k}=H_{\chi,r}(y_{p,k})$ 
and introduce 
\begin{equation}\label{def_uj_zj}
u_j=\omega_{j}  e^{-\pi x_j^2 + 2 \pi \theta x_j/\sqrt{q}}, \;\;\;
z_j=e^{-4 \pi \theta \frac{x_j}{(2p-1) \sqrt{q}}}. 
\end{equation}
The quadrature conditions \eqref{quadrature_condition0} then become
\begin{equation}\label{quadrature_condition3}
\sum\limits_{j=1}^p u_j z_j^k = \mu_{p,k}, \;\;\; k=0,1,\dots,2p-1. 
\end{equation}
This is a finite moment problem. Its solution can be constructed by the algebraic methods underlying Gaussian quadrature for complex functionals \cite{GONZALEZV1995,Laurie_2001,MILOVANOVIC2005,Pozza2016}. The usual positive-measure existence theorem does not apply here. A method for solving \eqref{quadrature_condition3} using orthogonal polynomials and a three-term recurrence relation is described in detail in \cite{Kuznetsov_2026}.  For completeness, we present an alternative construction.  Suppose that the Hankel matrix
\[
\mathsf H_p=(\mu_{p,j+k})_{j,k=0}^{p-1}
\]
is nonsingular. Find $h_0,\ldots,h_{p-1}$ by solving the system of linear equations 
\begin{equation*}
\sum_{\ell=0}^{p-1}h_\ell\mu_{p,k+\ell}=-\mu_{p,k+p},
\qquad k=0,\ldots,p-1.
\end{equation*}
If $R_p(z):=z^p+\sum_{\ell=0}^{p-1}h_\ell z^\ell$ has $p$ distinct nonzero roots, take these roots as the nodes $z_j$ and determine $u_j$ from the first $p$ equations in \eqref{quadrature_condition3}. The recurrence encoded by $R_p$ then gives the remaining equations. Under these hypotheses, the solution $u_j$, $z_j$ to \eqref{quadrature_condition3} is unique up to permutation.

Once a solution of \eqref{quadrature_condition3} has been obtained, we order the pairs $(u_j,z_j)$ so that  the moduli $|z_j|$ are non-decreasing.  We recover the quadrature coefficients $\omega_j$ and $x_j$ from \eqref{def_uj_zj}: 
\[
x_j=-\frac{(2p-1)\sqrt{q}}{4\pi \theta} \ln(z_j), \;\;\;
\omega_j=u_j e^{\pi x_j^2 - 2\pi \theta x_j/\sqrt{q}}, \;\;\;
j=1,2,\dots,p. 
\]
With the coefficients $\{\omega_j\}_{1\le j\le p}$ and $\{x_j\}_{1\le j \le p}$ computed, we now use \eqref{eqn_quadrature_integral} and approximate 
\[
{\mathcal I}_{\chi,r}(s,M)=\int_{\r} e^{g(x;s)} \eta(\d x) 
\approx 
\sum\limits_{j=1}^p \omega_j e^{g(x_j;s)}=
\sum\limits_{j=1}^p \omega_j e^{-2\pi M \theta x_j/\sqrt{q}}
\Big(1+\frac{\sqrt{q}x_j}{\theta M}  \Big)^{-s}.
\]
Finally, set $\omega^{(r)}_{\chi,p,j}=\omega_j$ and $\lambda^{(r)}_{\chi,p,j}=\theta \sqrt{q} x_j$. The above formula, combined with  \eqref{eqn_f_chi_main} and \eqref{def_I_chi_r}, gives us the desired approximation
\eqref{f_chi_approximation} to $f_{\chi}(s)$.

A useful consistency check is to verify the original $2p$ interpolation conditions in terms of the stored coefficients: for $k=0,\ldots,2p-1$
 \begin{equation*}
 \sum\limits_{j=1}^p \omega^{(r)}_{\chi,p,j} e^{-\frac{\pi \i}{q} (\lambda^{(r)}_{\chi,p,j})^2-\frac{2 \pi}{q} y_{p,k} \lambda^{(r)}_{\chi,p,j}}=H_{\chi,r}(y_{p,k}).
 \end{equation*}

The observed coefficient symmetries \eqref{lambda_omega_symmetry1} and \eqref{lambda_omega_symmetry2} correspond to the kernel identities below. To transfer reflection symmetry from kernels to coefficients $\omega^{(r)}_{\chi,p,j}$ and $\lambda^{(r)}_{\chi,p,j}$, one would need nondegeneracy of the moment problem. We retain \eqref{lambda_omega_symmetry1} and \eqref{lambda_omega_symmetry2} as observations for the reported coefficient tables.
 
 \begin{proposition}\label{prop_symmetry}
${}$ \begin{itemize}
\item[(i)] For all $0\le r<q$
\begin{equation}\label{K_symmetry_1}
K_{\chi,q-1-r}(-x)=\chi(-1) K_{\chi,r}(x), \;\;\; x\in \r. 
\end{equation}
\item[(ii)]
 When $4|q$, we have for all $0\le r < q/2$
\begin{equation}\label{K_symmetry_2}
K_{\chi,r+q/2}(x)=\i (-1)^r  K_{\chi,r}(x), \;\;\; x\in \r. 
\end{equation}
\end{itemize}
\end{proposition}
\begin{proof}
The definition \eqref{def_W_chi} gives the quasiperiodicity and parity identities
\begin{equation*}
W_{\chi}(u+q)=(-1)^q W_{\chi}(u), \;\;\;
W_{\chi}(-u)=-\chi(-1) W_{\chi}(u). 
\end{equation*}
Substituting these identities into \eqref{def_K_chi_r} proves \eqref{K_symmetry_1}. To prove \eqref{K_symmetry_2}, suppose that $4|q$. Then $\chi(1+q/2)=-1$; otherwise, $\chi$ would factor through modulus $q/2$, contradicting its primitivity.  Every $n$ coprime to $q$ is odd, and hence
\[
n+\frac{q}{2} \equiv n \Big(1+\frac{q}{2} \Big) 
\; (\textnormal{mod} \; q). 
\]
It follows that $\chi(n+q/2)=-\chi(n)$ for every integer $n$, including those not coprime to $q$, for which both sides vanish. Shifting the summation index in \eqref{def_W_chi} by $q/2$ then yields
\[
W_{\chi}\Big(u+\frac{q}{2}\Big)=e^{-\frac{\pi \i q}{4}} W_{\chi}(u).
\] 
Combining this identity with \eqref{def_K_chi_r} gives  \eqref{K_symmetry_2}. 
\end{proof}

\section{Computing $L(s,\chi)$ in the entire complex plane}\label{Section_computing_L}

To evaluate $L(s,\chi)$ throughout the complex plane, it suffices to evaluate both $L(s,\chi)$ and $L(s,\bar\chi)$ in the region $\sigma\ge1/2$, $t\ge0$. The functional equation \eqref{reflection_formula} and the conjugation identity
\[
L(\bar s,\chi)=\overline{L(s,\bar\chi)}
\]
then determine the remaining values. In the basic region $\sigma\ge1/2$, $t\ge0$, we select among direct summation, Euler--Maclaurin summation, and our quadrature approximation according to their estimated computation cost. Figure \ref{fig2} illustrates a partition with fixed thresholds; the implementation may instead select the method  according to its estimated computation cost at the given  value of $s$.

\begin{figure}[t!]
\centering
\includegraphics[height =7.0cm]{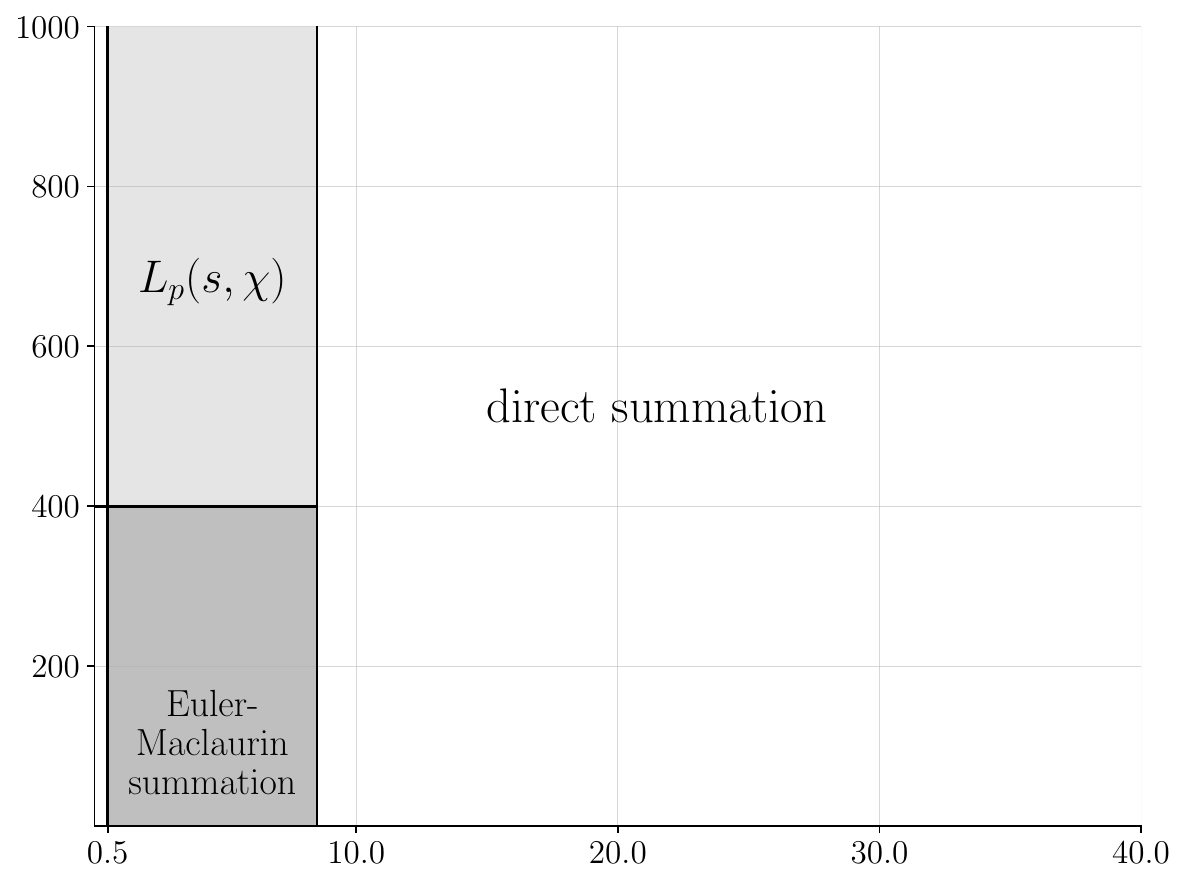}
\caption{Illustrative division of the region $\sigma\ge1/2$, $t\ge0$ among the three methods for evaluating $L(s,\chi)$. The horizontal switch is $C_2=400$; the vertical line represents a choice of $C_1$.} 
\label{fig2}
\end{figure}  

Direct summation is useful when $\sigma\ge C_1>1$. For an integer $N\ge1$, write
\begin{equation*}
L(s,\chi)=\sum\limits_{n=1}^N \frac{\chi(n)}{n^{s}}+
{\mathcal E}_{\chi}(N,s),
\end{equation*}
where 
\[
|{\mathcal E}_{\chi}(N,s)|
\le\sum_{n>N}\frac{|\chi(n)|}{n^{\sigma}}
\le\sum_{n>N}n^{-\sigma}
<\frac{N^{1-\sigma}}{\sigma-1}
\le\frac{N^{1-C_1}}{C_1-1}.
\]

For $1/2\le\sigma<C_1$ and $0\le t\le C_2$, illustrated by the darker shaded region in Figure \ref{fig2}, we use Euler--Maclaurin summation. Let $N$ be a positive multiple of $q$, and set $\delta:=1-\kappa$. Thus $\delta=0$ for odd characters and $\delta=1$ for even characters. We use the rising factorial notation $(s)_0:=1$ and $(s)_n:=s(s+1)\cdots(s+n-1)$ for $n\ge1$. For fixed $s$, $\chi$, and $J\ge1$, as $N\to\infty$ through multiples of $q$,
\begin{equation*}
L(s,\chi)=\sum\limits_{n=1}^N \frac{\chi(n)}{n^s}+
N^{-s} \sum\limits_{j=0}^{J-1} b_j \frac{(s)_{2j+\delta}}{N^{2j+\delta}}+O_{s,\chi,J}(N^{-\re(s)-2J-\delta}).
\end{equation*}
The coefficients $b_j$ are given by
\[
b_j=\frac{(-1)^{\delta+1} q^{2j+\delta}}{(2j+\delta+1)!}
\sum\limits_{\ell=1}^q \chi(\ell) B_{2j+\delta+1}(\ell/q).
\]
Here $B_n(x)$ denotes the Bernoulli polynomial. This expansion follows from the character analogue of Euler--Maclaurin summation \cite[Theorem 4.1]{BERNDT1975}. It can also be derived from the Hurwitz-zeta asymptotic expansion \cite[Eq.\ (25.11.43)]{NIST} and the finite decomposition \cite[Eq.\ (25.15.3)]{NIST}. 

In the implementation, we choose $N=q\lfloor\alpha|s|+\beta\rfloor$, with positive constants $\alpha$ and $\beta$ chosen for each character so that the desired accuracy is attained in the region $1/2 \le \sigma < C_1$ and $0\le t \le C_2$. 

Finally, for $1/2\le\sigma<C_1$ and $t>C_2$, illustrated by the lighter shaded region in Figure \ref{fig2}, we use $L_p(s,\chi)$.

We implemented the combined algorithm for $L(s,\chi)$ and $L'(s,\chi)$ for all primitive characters with $3\le q\le10$. Our approximation to $L'(s,\chi)$ is obtained by differentiating the right-hand side of \eqref{L(s,chi)_approximation} with respect to $s$ while holding $N$ fixed (see the corresponding definition of $\zeta_p^{(1)}(s)$ in \cite{Kuznetsov_2026}).  The coefficients and Fortran modules can be downloaded from \href{https://github.com/Alexey-Kuznetsov-math/Dirichlet-L-functions}{GitHub} or from the author's \href{https://kuznetsovmath.ca/}{webpage}. 

All calculations in these modules use quadruple precision, with $p=40$ and $C_2=400$. For $0\le t\le400$, the implementation selects between Euler--Maclaurin and direct summation according to their estimated cost. For $t>400$, it selects between $L_{40}(s,\chi)$ and direct summation.
Comparisons with higher-precision implementations of $L(s,\chi
)$ and $L'(s,\chi)$ based on Euler--Maclaurin summation and MPFUN2020 \cite{Bailey_2020} gave typical relative errors of order $10^{-31}$ for $|t|<100$, $10^{-30}$ for $|t|<10^3$, and $10^{-29}$ for $|t|<10^4$.


\end{document}